\documentclass[a4paper]{amsart}

\usepackage{amsfonts, latexsym, amssymb, graphicx,
multicol, mathrsfs, color, amscd, verbatim, paralist,
xspace, url, euscript, stmaryrd,  amsmath, amscd, enumitem,
bbold, multirow, tikz, mathtools, mathabx}
\usepackage[all,pdf,cmtip]{xy}

\usepackage[colorlinks, linkcolor=blue, citecolor=magenta, urlcolor=cyan]{hyperref}

\usepackage{mathabx}

\def\ra{\rightarrow}

\newtheorem{theorem}{THEOREM}[section]

\newtheorem{proposition}[theorem]{Proposition}
\newtheorem{lemma}[theorem]{Lemma}

\theoremstyle{definition}

\theoremstyle{remark}
\newtheorem{remark}[theorem]{Remark}

\newcommand\CC{{\mathbb C}}

\newcommand\RR{{\mathbb R}}

\def\SO{\mathop{\rm SO}\nolimits}

\def\Re{\mathop{\rm Re}\nolimits}
\def\Im{\mathop{\rm Im}\nolimits}

\makeatletter
\def\blfootnote{\xdef\@thefnmark{}\@footnotetext}
\makeatother

\begin{document}

\title[Homogeneous Hypersurfaces in $\CC^3$]{On homogeneous 
hypersurfaces in $\CC^3$}
\author[Isaev]{Alexander Isaev}

\address{Mathematical Sciences Institute\\
Australian National University\\
Acton, ACT 2601, Australia}
\email{alexander.isaev@anu.edu.au}

\maketitle

\thispagestyle{empty}

\pagestyle{myheadings}

\begin{abstract} 
We consider a family $M_t^n$, with $n\ge 2$, $t>1$, of real hypersurfaces in a complex affine $n$-dimensional quadric arising in connection with the classification of homogeneous compact simply-connected real-analytic hypersurfaces in\, $\CC^n$ due to Morimoto and Nagano. To finalize their classification, one needs to resolve the problem of the embeddability of $M_t^n$ in\, $\CC^n$ for $n=3,7$. In our earlier article we showed that $M_t^7$ is not embeddable in\, $\CC^7$ for every $t$ and that $M_t^3$ is embeddable in\, $\CC^3$ for all $1<t<1+10^{-6}$. In the present paper, we improve on the latter result by showing that the embeddability of $M_t^3$ in fact takes place for $1<t<\sqrt{(2+\sqrt{2})/3}$. This is achieved by analyzing the explicit totally real embedding of the sphere $S^3$ in $\CC^3$ constructed by Ahern and Rudin. For $t\ge\sqrt{(2+\sqrt{2})/3}$ the problem of the embeddability of $M_t^3$ remains open.
\end{abstract}

\section{Introduction}\label{intro}
\setcounter{equation}{0}

For $n\ge 2$, consider the $n$-dimensional affine quadric in $\CC^{n+1}$:
\begin{equation}
Q^n:=\{(z_1,\dots,z_{n+1})\in\CC^{n+1}:z_1^2+\dots+z_{n+1}^2=1\}.\label{quadric}
\end{equation}
The group $\SO(n+1,\RR)$ acts on $Q^n$, with the orbits of the action being the sphere $S^n=Q^n\cap\RR^{n+1}$ as well as the compact pairwise CR-nonequivalent strongly pseudoconvex hypersurfaces
\begin{equation}
M_t^n:=\{(z_1,\dots,z_{n+1})\in\CC^{n+1}: |z_1|^2+\dots+|z_{n+1}|^2=t\}\cap Q^n,\,\, t>1,\label{mtn}
\end{equation}
which are simply-connected for $n\ge 3$. The codimension 1 orbits $M_t^n$ play an important role in the classical paper \cite{MN} whose main objective was the determination of all compact simply-connected real-analytic hypersurfaces in $\CC^n$ homogeneous under an action of a Lie group by CR-transfor\-ma\-tions. Namely, it was shown in \cite{MN} that every such hypersurface is CR-equivalent to either the sphere $S^{2n-1}$ or, for $n=3,7$, to the manifold $M_t^n$ for some $t$. However, the question of the existence of a real-analytic CR-embedding of $M_t^n$ in $\CC^n$ for $n=3,7$ was not resolved, thus the classification in these two dimensions was not completed.

The family $M_t^n$ was studied in our earlier paper \cite{I}. In particular, in \cite[Corollary 2.1]{I} we observed that a necessary condition for the existence of a real-analytic CR-embedding of $M_t^n$ in $\CC^n$ is the embeddability of the sphere $S^n$ in $\CC^n$ as a totally real submanifold. The problem of the existence of a totally real embedding of $S^n$ in $\CC^n$ was considered by Gromov (see \cite{G1} and \cite[p.~193]{G2}), Stout-Zame (see \cite{SZ}), Ahern-Rudin (see \cite{AR}), Forstneri\v c (see \cite{F1}, \cite{F2}, \cite{F3}). It has turned out that $S^n$ admits a smooth totally real embedding in $\CC^n$ only for $n=3$, hence \cite[Corollary 2.1]{I} implies, in particular, that $M_t^7$ cannot be real-analytically CR-embedded in $\CC^7$. On the other hand, since $S^3$ is a totally real submanifold of $Q^3$, any real-analytic totally real embedding of $S^3$ in $\CC^3$ (which is known to exist, for instance, by \cite{AR}) extends to a biholomorphic map defined in a neighborhood of $S^3$ in $Q^3$. Owing to the fact that $M_t^3$ accumulate to $S^3$ as $t\ra 1$, this observation implies that $M_t^3$ admits a real-analytic CR-embedding in $\CC^3$ for all $t$ sufficiently close to 1. Thus, the classification of homogeneous compact simply-connected real-analytic hypersurfaces in complex dimension 3 is special as it includes manifolds other than the sphere $S^5$.

More precisely, in \cite[Theorem 3.1]{I} we showed that $M_t^3$ embeds in $\CC^3$ for all $1<t<1+10^{-6}$. This was proved by analyzing the holomorphic continuation\linebreak $F:\CC^4\ra\CC^3$ of the explicit polynomial totally real embedding of $S^3$ in $\CC^3$ constructed in \cite{AR}. Also, in \cite[Conjecture 3.1]{I} we stated that the map $F$ should in fact yield an embedding for all $1<t<\sqrt{(2+\sqrt{2})/3}$. In the present paper we confirm the conjecture and therefore obtain the following result:

\begin{theorem}\label{main1} 
The hypersurfaces $M_t^3$ admit a real-analytic CR-embedding in $\CC^3$ for $1<t<\sqrt{(2+\sqrt{2})/3}$.
\end{theorem}

Before proceeding with the proof of the theorem we note that, although the problem of the embeddability of $M_t^3$ remains open for $t\ge \sqrt{(2+\sqrt{2})/3}$, one might be able to approach it by utilizing other totally real embeddings of $S^3$ in $\CC^3$ found in \cite{AR} (see Remark \ref{othermaps}). 

{\bf Acknowledgment.} This work is supported by the Australian Research Council.

\section{Proof of Theorem \ref{main1}}\label{sect2}
\setcounter{equation}{0}

As stated in the introduction, our argument is based on analyzing the holomorphic continuation of the explicit totally real embedding of $S^3$ in $\CC^3$ constructed in \cite{AR}. Let $(z,w)$ be coordinates in $\CC^2$ and let $S^3$ be realized in the standard way as the subset of $\CC^2$ given by
$$
S^3=\{(z,w)\in\CC^2: |z|^2+|w|^2=1\}.
$$
The Ahern-Rudin map, which embeds $S^3$ in $\CC^3$ as a totally real submanifold, is defined on all of $\CC^2$ as follows:
\begin{equation}
f:\CC^2\ra\CC^3,\quad f(z,w):=(z,w,w\overline{z}\overline{w}^2+iz\overline{z}^2\overline{w}).\label{mapf}
\end{equation}

Now, consider $\CC^4$ with coordinates $z_1,z_2,z_3,z_4$ and embed $\CC^2$ in $\CC^4$ as the totally real subspace $\RR^4$:
$$
(z,w)\mapsto(\hbox{Re}\,z,\hbox{Im}\,z,\hbox{Re}\,w,\hbox{Re}\,w).
$$
Clearly, the push-forward of the polynomial map $f$ extends from $\RR^4$ to a holomorphic map $F:\CC^4\ra\CC^3$ by the formula
\begin{equation}
\begin{array}{l}
\displaystyle F(z_1,z_2,z_3,z_4):=\Bigl(z_1+iz_2,z_3+iz_4,\\
\hspace{1.1cm}\displaystyle(z_3+iz_4)(z_1-iz_2)(z_3-iz_4)^2+i(z_1+iz_2)(z_1-iz_2)^2(z_3-iz_4)\Bigr).
\end{array}\label{extenmap}
\end{equation}
It will be convenient for us to argue in the coordinates
\begin{equation}
w_1:=z_1+iz_2,\,\,w_2:=z_1-iz_2,\,\,w_3:=z_3+iz_4,\,\,w_4:=z_3-iz_4.\label{coordw}
\end{equation}
In these coordinates 
the quadric $Q^3$ takes the form
\begin{equation}
\left\{(w_1,w_2,w_3,w_4)\in\CC^4: w_1w_2+w_3w_4=1\right\}\label{qnew}
\end{equation}
(see (\ref{quadric})), the sphere $S^3\subset Q^3$ the form
$$
\left\{(w_1,w_2,w_3,w_4)\in\CC^4: w_2=\bar w_1,\,w_4=\bar w_3\right\}\cap Q^3,\label{sphereq}
$$
the hypersurface $M_t^3\subset Q^3$ the form
\begin{equation}
\left\{(w_1,w_2,w_3,w_4)\in\CC^4: |w_1|^2+|w_2|^2+|w_3|^2+|w_4|^2=2t\right\}\cap Q^3\label{formm7}
\end{equation}
(see (\ref{mtn})), and the map $F$ the form
\begin{equation}
\displaystyle (w_1,w_2,w_3,w_4)\mapsto\Bigl(w_1,w_3,w_2w_3w_4^2+iw_1w_2^2w_4\Bigr)\label{map}
\end{equation}
(see (\ref{extenmap})).

Clearly, $F$ yields an embedding of $M_t^3$ in $\CC^3$ if its restriction $\tilde F:=F|_{Q^3}$ is nondegenerate and injective on $M_t^3$, therefore it is important to investigate the nondegeneracy and injectivity properties of $\tilde F$. 

Let us start with nondegeneracy. In \cite[Proposition 3.1]{I} we gave a necessary and sufficient condition for $\tilde F$ to degenerate somewhere on $M_t^3$. As the argument is quite short, we repeat it here--with additional details--for the sake of the completeness of our exposition.      

\begin{proposition}\label{embed} The map $\tilde F$ degenerates at some point of $M_t^3$ if and only if\linebreak $t\ge\sqrt{(2+\sqrt{2})/3}$.  
\end{proposition}

\begin{proof}
Observe that $|w_1|+|w_3|>0$ on $Q^3$. For $w_1\ne 0$ we choose $w_1,w_3,w_4$ as local coordinates on $Q^3$ and write the third component of $\tilde F$ as
$$
\varphi:=\frac{1-w_3w_4}{w_1}(iw_4+(1-i)w_3w_4^2)
$$
(see (\ref{qnew}), (\ref{map})). Then the Jacobian $J_{\tilde F}$ of $\tilde F$ is equal to
\begin{equation}
\frac{\partial\varphi}{\partial w_4}=\frac{(3i-3)w_3^2w_4^2+(2-4i)w_3w_4+i}{w_1},\label{phieq}
\end{equation}
hence it vanishes if and only if
\begin{equation}
w_3w_4=\frac{3\pm\sqrt{2}-i}{6}.\label{degenpoint}
\end{equation}
At such points using (\ref{qnew}) we obtain
\begin{equation}
|w_1|^2+|w_2|^2+|w_3|^2+|w_4|^2=|w_1|^2+\frac{2\mp\sqrt{2}}{6|w_1|^2}+|w_3|^2+\frac{2\pm\sqrt{2}}{6|w_3|^2}.\label{est1}
\end{equation}

Analogously, if $w_3\ne 0$, we choose $w_1,w_2,w_3$ as local coordinates on $Q^3$ and write the third component of $\tilde F$ as
$$
\psi:=\frac{1-w_1w_2}{w_3}(w_2+(i-1)w_1w_2^2)
$$
(see (\ref{qnew}), (\ref{map})). Then
$$
J_{\tilde F}=-\frac{\partial\psi}{\partial w_2}=-\frac{(3-3i)w_1^2w_2^2+(2i-4)w_1w_2+1}{w_3},
$$
which vanishes if and only if
$$
w_1w_2=\frac{3\pm\sqrt{2}+i}{6}.
$$
Hence for all points of degeneracy of $\tilde F$ we have $w_1\ne 0$, $w_3\ne 0$, and therefore such points are described as the zeroes of ${\partial\varphi}/{\partial w_4}$ or, equivalently, as the zeroes of ${\partial\psi}/{\partial w_2}$.

We now need the following elementary lemma, which we state without proof:

\begin{lemma}\label{min}
For fixed $p>0$, $q>0$, let
$$
g(x,y):=x+\frac{p}{x}+y+\frac{q}{y},\quad x>0,\,y>0.
$$
Then $\min_{x>0,y>0} g(x,y)=2(\sqrt{p}+\sqrt{q})$.
\end{lemma}

Set $p=(2+\sqrt{2})/6$, $q=(2-\sqrt{2})/6$ and observe that $\sqrt{p}+\sqrt{q}=\sqrt{(2+\sqrt{2})/3}$. From formulas (\ref{formm7}), (\ref{est1}) combined with Lemma \ref{min} one now easily deduces that if $J_{\tilde F}$ vanishes at a point of $M_t^3$, then $t\ge\sqrt{(2+\sqrt{2})/3}$. Conversely, for any $t\ge\sqrt{(2+\sqrt{2})/3}$ there exist $x_0>0$, $y_0>0$ such that $g(x_0,y_0)=2t$. It then follows from (\ref{qnew}), (\ref{formm7}), (\ref{degenpoint}) that the point
$$
W_0:=\left(\sqrt{x_0},\frac{3+\sqrt{2}+i}{6\sqrt{x_0}},\sqrt{y_0},\frac{3-\sqrt{2}-i}{6\sqrt{y_0}}\right)
$$
lies in $M_t^3$ and that $J_{\tilde F}$ vanishes at $W_0$. The proof is complete. \end{proof}

Next, we study the fibers of $\tilde F$. 

\begin{proposition}\label{propfibers}
Let two points $W=(w_1,w_2,w_3,w_4)$ and $\hat W=(\hat w_1,\hat w_2,\hat w_3,\hat w_4)$ lie in $Q^3$ and assume that $\tilde F(W)=\tilde F(\hat W)$. Then the following holds:
\begin{itemize}
\item[{\rm (a)}] $\hat w_1=w_1$, $\hat w_3=w_3$;

\item[{\rm (b)}] if $w_1=0$ or $w_3=0$, then $\hat W=W$;

\item[{\rm (c)}] if $w_1\ne 0$ and $w_3\ne 0$, then either $\hat W=W$ or
\begin{equation}
\hat w_4=\frac{2i-1+(1-i)w_3w_4+\sqrt{6iw_3^2w_4^2-(2+6i)w_3w_4+1}}{(2i-2)w_3};\label{solquadr}
\end{equation}

\item[{\rm (d)}] neither of the two values in the right-hand side of {\rm (\ref{solquadr})} is equal to $w_4$ if $W\in M_t^3$ for $t<\sqrt{(2+\sqrt{2})/3}$;

\item[{\rm (e)}] the two values in the right-hand side of {\rm (\ref{solquadr})} are distinct if $W\in M_t^3$ for $t<\sqrt{(2+\sqrt{2})/3}$.

\end{itemize}
Hence, the fiber $\tilde F^{-1}(\tilde F(W))$ consists of at most three points, and, if $W\in M_t^3$ with $w_1\ne 0$, $w_3\ne 0$ for $t<\sqrt{(2+\sqrt{2})/3}$, it consists of exactly three points.

\end{proposition}

\begin{proof} Part (a) is immediate from (\ref{map}). Furthermore, (\ref{map}) yields
\begin{equation}
\hat w_2w_3\hat w_4^{2}+iw_1\hat w_2^{2}\hat w_4=w_2w_3w_4^2+iw_1w_2^2w_4,\label{mainids}
\end{equation}
which together with (\ref{qnew}) implies part (b).    

From now on, we assume that $w_1\ne 0$ and $w_3\ne 0$. Then using (\ref{qnew}) we substitute
\begin{equation}
w_2=\frac{1-w_3w_4}{w_1},\quad \hat w_2=\frac{1-w_3\hat w_4}{w_1}\label{expr234}
\end{equation}
into (\ref{mainids}) and simplifying the resulting expression obtain
\begin{equation}
\begin{aligned}[b]
&(\hat w_4-w_4)\Bigl[(i-1)w_3^2\,\hat w_4^2+ \Bigl((1-2i)w_3+(i-1)w_3^2w_4\Bigr)\hat w_4+\\
&\hspace{4cm}\Bigl(i+(1-2i)w_3w_4+(i-1)w_3^2w_4^2\Bigr)\Bigr]=0.
\end{aligned}\label{mainids4}
\end{equation}
We treat identity (\ref{mainids4}) as an equation with respect to $\hat w_4$. By part (a) and formula (\ref{expr234}), the solution $\hat w_4=w_4$ leads to the point $W$. Further, the solutions of the quadratic equation  
\begin{equation}
\begin{aligned}[b]
&(i-1)w_3^2\,\hat w_4^2+ \Bigl((1-2i)w_3+(i-1)w_3^2w_4\Bigr)\hat w_4+\\
&\hspace{3cm}\Bigl(i+(1-2i)w_3w_4+(i-1)w_3^2w_4^2\Bigr)=0
\end{aligned}\label{qudreq}
\end{equation}
are given by formula (\ref{solquadr}). This establishes part (c). 

Next, for $\hat w_4=w_4$ equation (\ref{qudreq}) becomes
$$
(3i-3)w_3^2w_4^2+(2-4i)w_3w_4+i=0,
$$
which by formula (\ref{phieq}) implies that the Jacobian $J_{\tilde F}$ vanishes at $W$. Therefore, part (d) follows from Proposition \ref{embed}.

Finally, suppose that $6iw_3^2w_4^2-(2+6i)w_3w_4+1=0$. Then we have
$$
w_3w_4=\frac{3\pm 2\sqrt{2}-i}{6}.
$$
At such points using (\ref{expr234}) we obtain
\begin{equation}
|w_1|^2+|w_2|^2+|w_3|^2+|w_4|^2=|w_1|^2+\frac{3\mp2\sqrt{2}}{6|w_1|^2}+|w_3|^2+\frac{3\pm2\sqrt{2}}{6|w_3|^2}.\label{est11}
\end{equation}
From formulas (\ref{formm7}), (\ref{est11}) combined with Lemma \ref{min} for $p=(3+2\sqrt{2})/6$,\linebreak $q=(3-2\sqrt{2})/6$ one now deduces that $t\ge 2/\sqrt{3}$. Since $2/\sqrt{3}>\sqrt{(2+\sqrt{2})/3}$, part (e) follows. The proof is complete. \end{proof}

Propositions \ref{embed}, \ref{propfibers} and formula (\ref{expr234}) imply that in order to establish Theorem \ref{main1} it suffices to show that for every value $t<\sqrt{(2+\sqrt{2})/3}$ and every point $W=(w_1,(1-w_3 w_4)/w_1,w_3,w_4)\in M_t^3$ with $w_1\ne 0$, $w_3\ne 0$, the point $\hat W:=\left(w_1,(1-w_3\hat w_4)/w_1,w_3,\hat w_4\right)$ does not lie in $M_t^3$ for any of the two choices of $\hat w_4$ in (\ref{solquadr}).

Let
\begin{equation}
{\mathcal E}:=\left\{z\in\CC: |1-z|+|z|<\sqrt{(2+\sqrt{2})/3}\right\}.\label{domainE}
\end{equation}
Clearly, ${\mathcal E}$ is the domain in $\CC$ bounded by the ellipse
$$
\left\{z\in\CC: |1-z|+|z|=\sqrt{(2+\sqrt{2})/3}\right\},
$$
which has foci at $1$ and $0$. We will need the following lemma:

\begin{lemma}\label{ellipse}
If a point $W=(w_1,(1-w_3 w_4)/w_1,w_3,w_4)$ with $w_1\ne 0$, $w_3\ne 0$ lies in $M_t^3$ for some $t<\sqrt{(2+\sqrt{2})/3}$, then $w_3w_4\in {\mathcal E}$. Moreover, for every $z\in {\mathcal E}$ there exists $t<\sqrt{(2+\sqrt{2})/3}$ and $W=(w_1,(1-w_3 w_4)/w_1,w_3,w_4)\in M_t^3$ with $w_1\ne 0$, $w_3\ne 0$ such that $z=w_3w_4$.
\end{lemma}

\begin{proof}
Fix $W=(w_1,(1-w_3 w_4)/w_1,w_3,w_4)\in M_t^3$ with $w_1\ne 0$, $w_3\ne 0$ and let $a:=w_3w_4$. Then (\ref{formm7}) yields
\begin{equation}
|w_1|^2+\frac{|1-a|^2}{|w_1|^2}+|w_3|^2+\frac{|a|^2}{|w_3|^2}=2t.\label{condona}
\end{equation}
By Lemma \ref{min} with $p=|1-a|^2$, $q=|a|^2$, the left-hand side of (\ref{condona}) is bounded from below by $2(|1-a|+|a|)$. Hence, if $t<\sqrt{(2+\sqrt{2})/3}$, one has $a\in {\mathcal E}$. 

Now, fix $z\in {\mathcal E}$ and assume first that $z\ne 0,1$. Set $t=|1-z|+|z|$ and $W=(\sqrt{1-z}, \sqrt{1-z}, \sqrt{z}, \sqrt{z})$, where we choose one of the two possible values of the square root arbitrarily. In each of the remaining cases $z=0$ and $z=1$ we let $t=1.045$ and set $W=(1,1,0.3,0)$ and $W=(0.3,0,1,1)$, respectively. It is then clear that $t<\sqrt{(2+\sqrt{2})/3}$, the point $W$ lies in $M_t^3$, the first and third coordinates of $W$ are nonzero, and the product of the third and fourth coordinates of $W$ is equal to $z$. \end{proof}

We will now finalize the proof of the theorem. Fix $t<\sqrt{(2+\sqrt{2})/3}$ and a point $W=(w_1,(1-w_3 w_4)/w_1,w_3,w_4)\in M_t^3$ with $w_1\ne 0$, $w_3\ne 0$. Consider the point $\hat W:=\left(w_1,(1-w_3\hat w_4)/w_1,w_3,\hat w_4\right)$ for some choice of $\hat w_4$ in (\ref{solquadr}). Assume that $\hat W\in M_t^3$ and let $a:=w_3w_4$, $\hat a:=w_3\hat w_4$. Then by (\ref{solquadr}) we see
\begin{equation}
\hat a=\frac{2i-1+(1-i)a+\sqrt{6ia^2-(2+6i)a+1}}{2i-2}.\label{solquadr1}
\end{equation}
Also, recalling that formula (\ref{solquadr}) is derived from equation (\ref{qudreq}), we have
\begin{equation}
(i-1)\hat a^2+ (1-2i)\hat a+(i-1)\hat a a+(i-1)a^2+(1-2i)a+i=0.\label{qudreq1}
\end{equation}

Next, analogously to (\ref{condona}), by (\ref{formm7}) the following holds:
\begin{equation}
\begin{array}{l}
\displaystyle |w_1|^2+\frac{|1-a|^2}{|w_1|^2}+|w_3|^2+\frac{|a|^2}{|w_3|^2}=2t,\\
\vspace{-0.1cm}\\
\displaystyle |w_1|^2+\frac{|1-\hat a|^2}{|w_1|^2}+|w_3|^2+\frac{|\hat a|^2}{|w_3|^2}=2t.
\end{array}\label{condonab}
\end{equation}
Subtracting one equation in (\ref{condonab}) from the other yields
\begin{equation}
\frac{|1-a|^2-|1-\hat a|^2}{|w_1|^2}+\frac{|a|^2-|\hat a|^2}{|w_3|^2}=0.\label{subtract}
\end{equation}
Let us now show that neither of the numerators in (\ref{subtract}) is zero. Indeed, the vanishing of any one of them implies the vanishing of the other, which can only occur if either $\hat a=a$ or $\hat a=\bar a$. The first possibility means that $\hat w_4=w_4$ contradicting part (d) of Proposition \ref{propfibers}. Assuming that $\hat a=\bar a$, from identity (\ref{qudreq1}) we deduce
\begin{equation}
(2i-2)\Re (a^2)+(2-4i)\Re a+(i-1)|a|^2+i=0.\label{hatbara}
\end{equation}  
Separating the real and imaginary parts of (\ref{hatbara}) and adding up the resulting identities, we see that $\Re a=1/2$, which leads to $(\Im a)^2+1/4=0$ thus implying a contradiction in this case as well.

Hence, by (\ref{subtract}) we have
$$
|w_1|^2=\frac{|1-a|^2-|1-\hat a|^2}{|\hat a|^2-|a|^2}|w_3|^2.
$$
Plugging this expression into any of the identities in (\ref{condonab}) and simplifying the resulting formulas, we obtain the following quadratic equation with respect to $|w_3|^2$:
\begin{equation}
A|w_3|^4-2t|w_3|^2+B=0,\label{quadratic}
\end{equation}
where
\begin{equation}
A:=\frac{2(\Re\hat a-\Re a)}{|\hat a|^2-|a|^2},\quad B:=\frac{(1-2\Re a)|\hat a|^2-(1-2\Re \hat a)|a|^2}{|1-a|^2-|1-\hat a|^2}.\label{AB}
\end{equation}
Taking into account that $t<\sqrt{(2+\sqrt{2})/3}$, for the discriminant $\Delta$ of the quadratic in the left-hand side of (\ref{quadratic}) we have
$$
\Delta=4(t^2-AB)<4\left(\frac{2+\sqrt{2}}{3}-AB\right).
$$

Notice now that, if we let $t$ vary in the interval $\left(1,\sqrt{(2+\sqrt{2})/3}\right)$ and $W$ in the hypersurface $M_t^3$, expressions (\ref{solquadr1}) and (\ref{AB}) yield an explicit formula for the product $AB$ viewed as a (two-valued) function of $a$ where, by Lemma \ref{ellipse}, the parameter $a$ varies in the domain ${\mathcal E}$ defined in (\ref{domainE}). A straightforward albeit tedious calculation, which we omit because of its length, now implies that everywhere on ${\mathcal E}$ one has $AB\ge(2+\sqrt{2})/3$. It then follows that $\Delta<0$, which contradicts the existence of a real (in fact, positive) root of equation (\ref{quadratic}). 

The proof of Theorem \ref{main1} is complete.  

\section{A few remarks}\label{sect3}
\setcounter{equation}{0}

We conclude the paper with several remarks.

\begin{remark}\label{greaterinj}
By \cite[Proposition 2.1]{I}, any real-analytic CR-embedding of $M_t^n$ in $\CC^n$ extends to a biholomorphic mapping of the domain
$$
\{(z_1,\dots,z_{n+1})\in\CC^{n+1}: |z_1|^2+\dots+|z_{n+1}|^2<t\}\cap Q^n
$$
onto a domain in $\CC^n$, where $n\ge 2$, $t>1$. Therefore, Theorem \ref{main1} implies that for every value $t<\sqrt{(2+\sqrt{2})/3}$ and every point $W=\left(w_1,(1-w_3w_4)/w_1,w_3,w_4\right)$ in $M_t^3$ with $w_1\ne 0$, $w_3\ne 0$, the point $\hat W:=\left(w_1,(1-w_3\hat w_4)/w_1,w_3,\hat w_4\right)$ lies in $M_{t'}^3$ where $t'\ge\sqrt{(2+\sqrt{2})/3}$ for each choice of $\hat w_4$ in (\ref{solquadr}).
\end{remark}

\begin{remark}\label{remarkinj}
It is easy to see that the injectivity of $\tilde F$ on $M_t^3$ fails for all sufficiently large values of $t$. For example, if $t\ge\sqrt{2}$, let $u\ne 0$ be a real number satisfying
$$
2u^2+\frac{1}{u^2}=2t, 
$$
(cf.~Lemma \ref{min}) and consider the following three distinct points in $Q^3$:
\begin{equation}
W_u:=\left(u,\frac{1}{u},u,0\right), W_u':=\left(u,0,u,\frac{1}{u}\right), W_u'':=\left(u,\frac{1+i}{2u},u,\frac{1-i}{2u}\right).\label{threepoints}
\end{equation}
Then $W_u, W_u', W_u''\in M_t^3$ and $\tilde F(W_u)=\tilde F(W_u')=\tilde F(W_u'')=(u,u,0)$. Since by Proposition \ref{propfibers} every fiber of $\tilde F$ contains at most three points, $W_u,W_u', W_u''$ form the complete fiber of $\tilde F$ over $(u,u,0)$. 
\end{remark}

\begin{remark}\label{othermaps}
In \cite{AR} the authors in fact introduced not just the map $f$ (see (\ref{mapf})) but a class of maps of the form
$$
g:\CC^2\ra\CC^3,\quad g(z,w):=(z,w,P(z,\bar z,w,\bar w)).
$$
Here $P$ is a harmonic polynomial given by
$$
P=\left(\bar z\frac{\partial}{\partial w}-\bar w\frac{\partial}{\partial z}\right)\left(\sum_{j=1}^m\frac{1}{p_j(q_j+1)}Q_j\right),
$$
where $Q_j$ is a homogeneous harmonic complex-valued polynomial on $\CC^2$ of total degree $p_j\ge 1$ in $z,w$ and total degree $q_j$ in $\bar z$, $\bar w$, such that the sum $Q:=Q_1+\dots+Q_m$ does not vanish on $S^3$. It is possible that one can investigate the embeddability of $M_t^3$ in $\CC^3$ for $t\ge\sqrt{(2+\sqrt{2})/3}$ using these more general maps. Note, however, that while it is tempting to take $Q$ to be a polynomial in $|z|^2$, $|w|^2$ (as was done in \cite{AR}), one should avoid doing so as otherwise the resulting polynomial $P$ would be divisible by $\bar z\bar w$, which implies that the holomorphic extension $G$ of the push-forward of $g$ to $\RR^4$ is not injective on $M_t^3$ with $t\ge\sqrt{2}$ (cf.~Remark \ref{remarkinj}). Indeed, writing $G$ in the coordinates $w_j$ defined in (\ref{coordw}), for the points $W_u$, $W_u'$ introduced in (\ref{threepoints}) one has $G(W_u)=G(W_u')=(u,u,0)$. Thus, one cannot obtain the embeddability of $M_t^3$ in $\CC^3$ for $t\ge\sqrt{2}$ by utilizing any of the maps introduced in \cite{AR}, with $Q$ being a function of $|z|^2$, $|w|^2$ alone. On the other hand, for $Q$ of a more general form an analysis of the holomorphic extension $G$ of the kind we performed above for the map $F$ becomes computationally quite challenging. 
\end{remark}

\begin{remark}\label{rem2}\rm As explained in \cite[Remark 2.2]{I}, every hypersurface $M_t^n$ is nonspherical. Therefore, every manifold $M_t^3$ embeddable in $\CC^3$ provides an example of a compact strongly pseudoconvex simply-connected hypersurface in $\CC^3$ without umbilic points. Such hypersurfaces have been known before, but the arguments required to obtain nonumbilicity for them are rather involved. For instance, the proof in \cite{W} of the fact that every generic ellipsoid in $\CC^n$ for $n\ge 3$ has no umbilic points relies on the Chern-Moser theory. Note for comparison that an example of a compact strongly pseudoconvex hypersurface in $\CC^2$ having no umbilic points has been constructed only very recently (see \cite{ESZ}).
\end{remark}

\end{document}